\documentclass[11pt,a4paper]{amsart}
\usepackage{amsmath,amsfonts}
\usepackage{graphicx}
\usepackage{amssymb}
\usepackage{amsthm}
\usepackage{yfonts}
\usepackage{kpfonts}

\usepackage[usenames,dvipsnames]{pstricks}
\usepackage{epsfig}
\usepackage{rotating}
\usepackage{pst-plot}
\usepackage{pst-eps}
\usepackage{pst-grad}
\usepackage{pstricks-add}
\usepackage{lmodern}
\usepackage{xcolor}
\usepackage{graphicx}
\usepackage[top=2cm,bottom=2cm,left=2cm,right=2cm,a4paper]{geometry}

\usepackage{etex}
\usepackage[all]{xy}

\usepackage{cite}

\def\eps{{\varepsilon}}

\newcommand{\cc}{{\mathbb C}}
\newcommand{\zz}{{\mathbb Z}}

\newcommand{\af}{{\mathbb A}}
\newcommand{\pp}{{\mathbb P}}

\newcommand{\II}{{\mathbb I}}

\newcommand{\spec}{\hbox{Spec }}

\newcommand{\Hom}{\hbox{Hom}}

\newcommand{\Pic}{\hbox{Pic}}
\newcommand{\Kod}{\hbox{Kod}}

\newcommand{\sh}{\mathscr}
\newcommand{\cal}{\mathcal} 

\newcommand{\Mg}{\mathcal{M}_g}
\newcommand{\Mgbar}{\overline{\mathcal{M}}_g}

\newcommand{\Fg}{\mathcal{F}_{g}}
\newcommand{\Fgn}{\mathcal{F}_{g,n}}

\theoremstyle{plain}
\newtheorem{theorem}{Theorem}[section]
\newtheorem{lemma}[theorem]{Lemma}
\newtheorem{proposition}[theorem]{Proposition}
\newtheorem{corollary}[theorem]{Corollary}

\theoremstyle{definition}
\newtheorem{remark}[theorem]{Remark}
\newtheorem{definition}[theorem]{Definition}

\title{Geometry of the moduli space of $n$-pointed K3 surfaces of genus 11}
\author{Ignacio Barros}
\address{
Humboldt-Universit\"{a}t zu Berlin\\
Institut f\"{u}r Mathematik\\
Unter den Linden 6, 10099 - Berlin\\ 
Germany.} 
\email[]{iabarros@math.hu-berlin.de}

\begin{document}

\maketitle

\begin{abstract}
We prove that the moduli space of polarized $K3$ surfaces of genus eleven with $n$ marked points is unirational when $n\leq 6$ and uniruled when $n\leq7$. As a consequence we settle a long standing but not proved assertion about the unirationality of $\cal{M}_{11,n}$ for $n\leq6$. We also prove that the moduli space of polarized $K3$ surfaces of genus eleven with $n\geq9$ marked points has non-negative Kodaira dimension.\\

\textbf{Mathematics Subject Classification (2010)}: 14J28, 14H10, 14J10, 14E08.
\end{abstract}


\nopagebreak



\section*{Introduction}

Let $\Fgn$ be the moduli space of tuples $(S,H,x_1,\ldots,x_n)$, where $(S,H)$ is a primitively polarized $K3$ surface of genus $g$ and $x_1,\ldots,x_n\in S$ are $n$ ordered marked points on $S$. Mukai in his celebrated series of papers \cite{M1, M2, M3, M4, M5} established structure theorems for $\cal{F}_g$ in the range $g\leq 12$ and $g=13,16,18,20$. Mukai's results imply the unirationality of $\cal{F}_g$ in this range. Farkas and Verra \cite[Thm. 1.1]{FV2} proved the rationality of $\cal{F}_{14,1}$ and therefore the unirationality of $\cal{F}_{14}$. On the other hand, it is known that $\cal{F}_g$ is of general type for $g>62$ and for some other values $g\geq 47$, cf. \cite[Thm. 1]{GHS}. The forgetful map 
$$u:\cal{F}_{g,n}\to\Fg$$ 
is a morphism fibered in Calabi-Yau varieties. By Iitaka's easy addition formula, 
$$\Kod(\cal{F}_{g,n})\leq \dim(\Fg)=19.$$ 
In particular, $\Fgn$ is never of general type when $n\geq 1$. Moreover, by \cite{Ka} 
$$\Kod(\cal{F}_{g,n+1})\geq \Kod(\Fgn).$$ 
Thus, when $\cal{F}_g$ is of general type, $\Kod(\Fgn)=19$. It is established in \cite[Thm. 5.1]{FV2} that $\cal{F}_{11,n}$ is unirational for $n=1$ and it has positive Kodaira dimension equal to $19$ for $n=11$. The question that concerns us is about the Kodaira dimension of $\cal{F}_{11,n}$ for $1< n\leq 10$. \\

In genus $g=11$, Mukai \cite{M6} proved that a general curve $C$ of genus $11$ has a unique $K3$ extension $S$ up to isomorphism with $\Pic(S)=\zz\cdot\left[C\right]$. This construction induces a rational map
$$\begin{array}{rcl}
\psi:\cal{M}_{11,n}&\dashrightarrow&\cal{F}_{11,n}\\
\left[C,x_1,\ldots,x_n\right]&\mapsto&\left(S,\cal{O}_S(C),x_1,\ldots,x_n\right),
\end{array}$$
which is dominant for $n\leq 11$ and birational for $n=11$.\\

It is claimed in \cite[Table 3]{L} that $\cal{M}_{11,n}$ is unirational for $n\leq 10$. But Logan's argument only establishes the uniruledness of $\cal{M}_{11,n}$ by means of the Mukai map $\psi$, which is birationally a $\pp^{11-n}$-bundle over $\cal{F}_{11,n}$. The birational description of the base is still missing when $n\leq 10$. The aim of this paper is to address this question. We prove the following theorem regarding the birational geometry of $\mathcal{F}_{11,n}$.

\begin{theorem}
\label{thm}
The moduli space $\cal{F}_{11,n}$ is unirational for $n\leq 6$ and uniruled for $n\leq 7$.
\end{theorem}

Using the already mentioned fibration $\psi:\cal{M}_{11,n}\dashrightarrow\cal{F}_{11,n}$ we obtain the following corollary.

\begin{corollary}
For $n\leq6$, the moduli space $\cal{M}_{11,n}$ is unirational.
\end{corollary}

In \cite{M6} it was proven that the space
$$\cal{V}_{g,0}=\left\{\right(S,H,C)\mid (S,H)\in \cal{F}_g, C\in|H|\hbox{ smooth}\}$$
is birational to $\Mg$ when $g=11$. Mukai's proof consists in constructing an explicit birational inverse for the forgetful map $p:(S,H,C)\mapsto \left[C\right]\in\cal{M}_{11}$. An alternative proof of the same fact was given by \cite[Prop. 4.4]{CLM}, where they showed that the fibers of $p$ are irreducible and, in the genus $11$ case, zero dimensional. Our results follows from the study of the map $p$ when the curve $C$ degenerates to an irreducible $\delta$-nodal curve. We call $\cal{V}_{g,\delta,l}$ the moduli of irreducible $\delta$-nodal curves with $l$ marked points on a polarized $K3$ surface;
$$\cal{V}_{g,\delta,l}:=\left\{(S,H,X,x_1,\ldots,x_\delta,y_1,\ldots,y_l)\left|\begin{aligned} 
&(S,H)\in\cal{F}_{g}, \\
&X\in|H|\hbox{ irreducible $\delta$-nodal at $x_1,\ldots,x_\delta$}\\
&\hbox{and }y_1,\ldots,y_l\in X.\end{aligned}\right\}\right.$$
The main diagram to consider is the following:
\begin{displaymath}
\xymatrix{
&\cal{V}_{g,\delta,l}\ar[dl]_{\pi}\ar[dr]^{c_{g,\delta,l}}&\\
\cal{F}_{g,\delta+l}&&\mathcal{M}_{11-\delta,2\delta+l}\big/\zz_{2}^{\oplus \delta}.
}
\end{displaymath}
\\
The map $\pi$ on the left forgets the nodal curve
$$\pi:(S,H,X,x_1,\ldots,x_{\delta},y_1,\ldots,y_l)\mapsto(S,H,x_1,\ldots,x_{\delta},y_1,\ldots,y_l)$$ 
and the map on the right is the one induced by normalization, defined as
$$c_{g,\delta,l}:(S,H,X,x_1,\ldots,x_{\delta},y_1,\ldots,y_l)\mapsto[C,p_1+q_1,\ldots,p_\delta+q_\delta, y_1,\ldots,y_l],$$
where $C\to X$ is the normalization map and $p_i+q_i$ is the preimage of the $i$-th marked node. Here the $\delta$ copies of $\zz_2$ permute the first $\delta$ pairs of points.\\ 

By deformation theoretic arguments we prove in \S 2 that  $c_{g,\delta,l}$ is dominant in the range $3\leq g\leq 11$, $\delta\leq g-2$, and $g\neq 10$. In the same section we also prove that $\pi$ is dominant in the same range, granted that $3\delta+l\leq g$.\\

In \S3 we show that for $g=11$ the map $c_{11,\delta,l}$ is finite and the general fiber is irreducible. This gives us the following intermediate result.

\begin{theorem}
\label{thm2}
The moduli map induced by the normalization 
$$\mathcal{V}_{g,\delta,l}\to\mathcal{M}_{g-\delta,2\delta+l}\big/\zz_{2}^{\oplus \delta}$$
is dominant for $3\leq g\leq 11$, $\delta\leq g-2$ and $g\neq 10$. Moreover, for $g=11$, the map above is a birational isomorphism. 
\end{theorem}

In the last section we show that, in the same range as above, when $3\delta+l=g$ the map $\pi$ is birational. In genus eleven this gives us the following result.

\begin{theorem}
\label{bir}
When $3\delta+l=11$, there exists a birational isomorphism 
$$\mathcal{F}_{11,\delta+l} \overset{\sim}{\dashrightarrow}\mathcal{M}_{11-\delta,2\delta+l}\big/\zz_2^{\oplus \delta}.$$
In particular $\mathcal{F}_{11,9}$ is birational to the $\zz_2$-quotient of $\mathcal{M}_{10,10}$.
\end{theorem}

It is known that the canonical class $K_{\overline{\mathcal{M}}_{10,10}}$ is effective, cf. \cite[Thm 1.9]{FP}. We show in the last section that the same holds for the $\zz_2$-quotient. This, together with the inequality $\hbox{Kod}(\Fgn)\leq \hbox{Kod}(\mathcal{F}_{g,n+1})$, gives us the following theorem.

\begin{theorem}
\label{thm3}
The Kodaira dimension of $\mathcal{F}_{11,n}$ is non-negative for $n\geq 9$.
\end{theorem} 

From the fibration $\psi: \mathcal{M}_{11,n}\dashrightarrow\mathcal{F}_{11,n}$ for $n\leq 11$ it follows:

\begin{corollary}
The moduli spaces $\mathcal{M}_{11,9}$ and $\mathcal{M}_{11,10}$ are not unirational.
\end{corollary}

\section*{Acknowledgements.}

Part of this work was carried out when I visited UGA in September 2017. I sincerely thank Scott Mullane and Benjamin Bakker for the invitation, hospitality, and mathematical suggestions. I am grateful to my PhD advisors Gavril Farkas and Rahul Pandharipande for their infinite patience, permanent support, and vast insight. Thanks also goes to an anonymous referees who pointed out style oversights and a crucial mathematical mistake in an older version of this paper. My PhD studies are generously supported by the Einstein Stiftung Berlin through the grant 8731100599.

\section{Nodal curves on $K3$ surfaces}
For positive integers $\delta\leq g$, the \textit{Severi variety} of irreducible $\delta$-nodal curves in the linear system $|H|$ is denoted by $V_{\delta}(S,H)$. It is well known that for $(S,H)\in\cal{F}_g$ general and $g\geq 2$, the space $V_\delta(S,H)$ is non-empty and each irreducible component is of dimension $g-\delta$. We refer to \cite{C1}, \cite[Cor. 1.2]{C2}, \cite{T}, and \cite{Fl} for fundamental facts on this matter.\\

In a similar way, we can mark the nodes. For a pointed polarized $K3$ surface $(S,H,x_1,\ldots, x_\delta)$ we denote by $V_{\delta}(S,H,x_1,\ldots,x_\delta)$ the space of irreducible $\delta$-nodal curves $X\in |H|$, with nodes at $x_1,\ldots,x_\delta$. Notice that 
$$\bigcup_{x_1,\ldots,x_\delta\in S}V_{\delta}(S,H,x_1,\ldots,x_\delta)\big/\Sigma_\delta\cong V_{\delta}(S,H),$$
where $\Sigma_\delta$ is the symmetric group of degree $\delta$.

\begin{definition}
For integers $g$ and $\delta$, such that $g\geq3$ and $0\leq\delta\leq g-2$, we define the \textit{universal Severi variety} $\cal{V}_{g,\delta}$ to be the algebraic stack whose coarse moduli space parameterizes tuples $(S,H,X,x_1,\ldots,x_\delta)$, where $(S,H,x_1,\ldots,x_\delta)\in\cal{F}_{g,\delta}$ and $X\in V_\delta(S,H,x_1,\ldots,x_\delta)$. There is a natural forgetful map
$$\mathcal{V}_{g,\delta}\to\Mgbar$$
that remembers the nodal curve. The moduli space $\mathcal{V}_{g,\delta,l}$ is defined as the fiber product 
$$\mathcal{V}_{g,\delta}\times_{\Mgbar}\overline{\mathcal{M}}_{g,l},$$
where $\overline{\mathcal{M}}_{g,l}\to\Mgbar$ is the map that forgets the marked points.
\end{definition}

The stack $\cal{V}_{g,\delta,l}$ is smooth and every irreducible component has dimension $19+l+(g-\delta)$. It was conjectured by Ciliberto and Dedieu \cite[Thm. 2.1]{CD} that the quotient $\cal{V}_{g,\delta}\big/\Sigma_\delta$ is always irreducible and they proved it in the range $3\leq g\leq 11$, $g\neq 10$ and $0\leq\delta\leq g$.  The forgetful map 
$$\pi_\delta:\cal{V}_{g,\delta}/\Sigma_\delta\to\cal{F}_g$$
is smooth and dominant when restricted to any irreducible component; see \cite[Prop. 4.8]{FlKPS}.\\

Let $(S,X,x_1,\ldots,x_\delta)\in\cal{V}_{g,\delta}$ be a $\delta$-nodal curve on a K3 and $\nu:C\to X$ the normalization map with $\nu^{*}(x_i)=p_i+q_i$. There are natural moduli maps to consider

\begin{displaymath}
\xymatrix{
&\cal{V}_{g,\delta}\ar[dl]^{\pi}\ar[r]^{c_{g,\delta}}\ar[dr]_{c_g}&\cal{M}_{g-\delta,[2\delta]}\ar[d]\\
\cal{F}_{g,\delta}&&\cal{M}_{g-\delta},
}
\end{displaymath}
where 
$$\cal{M}_{g-\delta,[2\delta]}:=\cal{M}_{g-\delta,2\delta}\big/ \zz_2^{\oplus \delta}.$$
Here each copy of $\zz_2$ permutes the pair of points $(p_i,q_i)$ and the map $c_{g,\delta}$ is defined by
$$c_{g,\delta}:(S,X,x_1,\ldots,x_\delta)\mapsto \left[C,p_1+q_1,\ldots,p_\delta+q_\delta\right].$$ 

It was proved in \cite[Thm. 5.1]{FlKPS} that for $3\leq g\leq 11$ and $0\leq \delta\leq g-2$, the map $c_g$ restricted to any irreducible component is dominant and the dimension of the general fiber is $22-2(g-\delta)$. This was generalized for higher genus in \cite{CFGK}. In previous work we were able to extend their result using similar techniques.

\begin{theorem}[Thm. 0.5 in \cite{Ba}]
For $3\leq g\leq 11$, $1\leq \delta\leq g-2$ and $g\neq 10$, the moduli map $c_{g,\delta}$ defined above is dominant when restricted to any irreducible component and the dimension of the general fiber is $22-2g$.
\end{theorem} 

As expected, when $g=11$ the map $c_{g,\delta}$ is generically finite and the situation is no different when we mark points on the nodal curve, since the map
$$c_{g,\delta,l}:\mathcal{V}_{g,\delta,l}\to \mathcal{M}_{g-\delta,[2\delta]+l}$$
is the fiber product
$$c_{g,\delta,l}=c_{g,\delta}\times_{\mathcal{M}_{g-\delta,[2\delta]}}q,$$
where $q$ is the forgetful map $\mathcal{M}_{g,[2\delta]+l}\to\mathcal{M}_{g,[2\delta]}$.\\

We recall the main arguments used in \cite{Ba} to prove the dominance and general fiber dimension of $c_{g,\delta}:\cal{V}_{g,\delta}\to \cal{M}_{g-\delta,[2\delta]}$ in the corresponding range.\\

The quotient map 
$$\cal{V}_{g,\delta}\to\cal{V}_{g,\delta}\big/\Sigma_\delta$$
is \'{e}tale and the tangent space of the quotient at a point $(S,X)$ corresponds to locally trivial first order deformations of the closed embedding $X\hookrightarrow S$. The deformation theory of such setting is governed by the sheaf $T_S\langle X\rangle$, defined to be the preimage of $T_X\subset T_S\mid_X$ under the restriction $T_S\to T_S\mid_X$, where $T_X$ is the tangent sheaf of the (nodal) curve $X$; see \cite[\S 3.4.4]{S}. More explicitly
$$T_{(S,X,x_1,\ldots,x_\delta)}\cal{V}_{g,\delta}\cong T_{(S,X)}\left(\cal{V}_{g,\delta}\big/\Sigma_\delta\right)\cong H^1(S,T_S\langle X\rangle)$$
and 
$$H^0(T_S\langle X\rangle)\cong H^2(T_S\langle X\rangle)=0.$$
The sheaf $T_S\langle X\rangle$ sits in the exact sequence
\begin{equation}
\label{ex.seq.0}
0\to T_S\langle X\rangle\to T_S\to \mathcal{N}'_{X/S}\to 0,
\end{equation}
where $\cal{N}'_{X/S}$ is the \textit{equisingular normal sheaf} of $X$ in S, cf. \cite[Prop. 1.1.9]{S}, whose zero cohomology group $H^0(\cal{N}'_{X/S})$ parameterizes locally trivial first order deformations of the closed embedding $X\subset S$, with $S$ fixed. \\

For a general point $(S,X,x_1,\ldots,x_\delta)\in\cal{V}_{g,\delta}$, let $\eps:\tilde{S}\to S$ be the blow up of $S$ at $x_1,\ldots,x_\delta$ and $E=E_1+\ldots+E_\delta$ the corresponding exceptional divisor. The normalization $f:C\to X\subset S$ is the restriction of $\eps$ to the proper transform of $X$ in $\tilde{S}$ and it lies in the linear system $|\eps^*H-2E|$. Consider the exact sequence
$$0\to T_C\to f^*T_S\to N_{f}\to0,$$
where $N_f$ is the normal sheaf of the map $f:C\to S$. Let $\lambda$ be the composition 
$$\lambda: \eps^*T_S\to f^*T_S\to N_f.$$ 
We call $\sh{F}_C$ the kernel of $\lambda$ and it sits in the following diagram:
\begin{equation}
\label{diagP5}
\xymatrix{
0\ar[r]&\sh{F}_C\ar[d]^{\tau}\ar[r]&\eps^*T_S\ar[d]\ar[r]^{\lambda}&N_f\ar@{=}[d]\ar[r]&0\\
0\ar[r]&T_C\ar[r]&f^*T_S\ar[r]&N_f\ar[r]&0.
}
\end{equation}

It is shown in \cite[Prop. 4.22]{FlKPS} that 
$$H^1(\tilde{S},\sh{F}_C)\cong H^1(S,T_S\langle X\rangle)$$
and $H^1(\tau)$ is the differential of the map $c:\cal{V}_{g,\delta}\to \cal{M}_{g-\delta}$. The space $\mathcal{V}_{g,\delta}$ in \cite{FlKPS} is $\mathcal{V}_{g,\delta}\big/\Sigma_\delta$ in our notation, but the results stays the same, as the quotient is \'{e}tale.\\

We proved \cite[Cor. 1.10 and Prop. 1.11]{Ba} that 
$$H^1(\tilde{S},\sh{F}_C(-E))\cong H^1(S,T_S\langle X\rangle)$$
and 
$$H^1(\tau(-E)):H^1(\tilde{S},\sh{F}_C(-E))\to H^1(T_C(-E\mid_C))$$
is the differential of $c_{g,\delta}$. When $(S,H)$ is general in $\cal{F}_g$, the cokernel of this map is isomorphic to $H^0(S,\Omega_S^1(H))$ that vanishes when $g\leq 11$ and $g\neq 10$, cf. \cite[\S5.2]{Be}. This establishes local dominance of 
$$c_{g,\delta}:\cal{V}_{g,\delta}\to\cal{M}_{g-\delta,[2\delta]}.$$

Finally when $g=11$, by dimension count the map 
$$c_{11,\delta,l}:\cal{V}_{11,\delta,l}\to\cal{M}_{11-\delta,[2\delta]+l}$$
is generically finite. In $\S3$ we show that the general fiber is irreducible.

\section{Deformation theory of pointed nodal curves on K3 surfaces}
The goal of this section is to show dominance of the map $c_{g,\delta,l}:\cal{V}_{g,\delta,l} \to\mathcal{M}_{g-\delta,[2\delta]+l}$ for $3\leq g\leq 11$, $0\leq \delta\leq g-2$, and $g\neq 10$. We also show the dominance of
$$\pi:\cal{V}_{g,\delta,l}\to\cal{F}_{g,\delta+l}$$
in the same range, when $3\delta+l\leq g$. With regard to the map $\pi$, notice that if we count dimensions naively, every marked point $y_1,\ldots,y_l\in S$ should impose one linear condition on the linear system $|H|$ and for a hyperplane section of $S\subset\pp^g$ to be nodal at $x_1,\ldots,x_\delta$ it has to contains all tangent $2$-planes at those points. Thus, every node should impose $3$ linear conditions on $|H|$. If the conditions imposed are independent, the map $\pi$ is expected be dominant when $3\delta+l\leq g$. 

\begin{definition}
Let $S$ be a K3 surface, $X\subset S$ a (nodal) curve and $y_1,\ldots,y_l\in X$ marked points on the curve, away from the nodes. We define 
$$T_S\langle X\mid y_1,\ldots,y_l\rangle\subset T_S\otimes \mathcal{I}_{y_1+\ldots+y_l}$$
to be the inverse image of $T_X(-y_1-\ldots-y_l)\subset \left(T_S\otimes \mathcal{I}_{y_1+\ldots+y_l}\right)\mid_X$ under the natural restriction
$$T_S\otimes \mathcal{I}_{y_1+\ldots+y_l}\to \left(T_S\otimes \mathcal{I}_{y_1+\ldots+y_l}\right)\mid_X.$$
The sheaf $T_S\langle X\mid y_1,\ldots,y_l\rangle$ is called the \textit{sheaf of germs of tangent vectors of the pointed K3 $(S,y_1,\ldots,y_l)$ which are tangent to the pointed curve $(X,y_1,\ldots,y_l)$.}
\end{definition}

To simplify the notation we write $y=y_1+\ldots+y_l$. The sheaf $T_S\langle X\mid y_1,\ldots,y_l\rangle$ sits in two exact sequences coming from restriction and inclusion respectively
\begin{equation}
\label{ex.seq.1}
\begin{aligned}
0\to T_S(-X)\otimes \mathcal{I}_{y}\to T_S\langle X\mid y_1,\ldots,y_l\rangle\to T_X(-y)\to0,
\end{aligned}
\end{equation}
\begin{equation}
\label{ex.seq.2}
\begin{aligned}
0\to T_S\langle X\mid y_1,\ldots,y_l\rangle\to T_S\otimes \mathcal{I}_{y}\to \mathcal{N}'_{X/S}(-y)\to0.
\end{aligned}
\end{equation}
The following proposition can be found in \cite[\S 3.4.4]{S} without the markings. The case of closed embeddings with markings is straightforward to extend, but for sake of completeness we give a proof.

\begin{proposition}
\label{prop}
Locally trivial first order deformations of the pointed closed embedding 
$$(y_1,\ldots,y_l\in X\hookrightarrow S)$$
are parameterized by $H^1\left(S, T_S\langle X\mid y_1,\ldots,y_l\rangle\right)$. The spaces $H^0$ and $H^2$ of the same sheaf parameterize local automorphisms an obstructions respectively and they both vanish when $g\geq 3$ and $\delta\leq g-2$. In particular, $H^1\left(S, T_S\langle X\mid y_1,\ldots,y_l\rangle\right)$ is the tangent space of $\mathcal{V}_{g,\delta,l}$ at the point $(S,H,X,x_1,\ldots,x_\delta,y_1,\ldots,y_l)$.
\end{proposition}

\begin{proof}
Let $\sh{X}\hookrightarrow \sh{S}$ be a locally trivial first order deformation of $X\hookrightarrow S$ and 
$$\tilde{y}_1,\ldots,\tilde{y}_l:\II\to\sh{X}$$ 
the sections corresponding to the markings. Here $\II$ denotes the scheme $\spec(\cc[\eps]/\eps^2)$. Let $\{U_i\}_{i\in I}$ be an affine open cover of $S$, 
$$\{V_i=X\cap U_i\}_{i\in I}$$ 
the induced affine open cover of $X$ and $\{W_i=\left(\cup_k \{y_k\}\right)\cap U_i\}$ the induced cover of the zero dimensional scheme $y=y_1+\ldots+y_l\subset X$. Notice that some $W_i$ might be empty. As usual, every locally trivial first order deformation is obtained by gluing the trivial deformations 
$$W_i\times \II\hookrightarrow V_i\times\II\hookrightarrow U_i\times \II$$ 
along the intersections $W_{ij}\times \II$, $V_{ij}\times \II$, and $U_{ij}\times \II$. This is equivalent to give local automorphisms $(A_{ij}, B_{ij}, C_{ij})$ such that the following diagram commutes:
\begin{displaymath}
\xymatrix{
W_{ij}\times\II \ar[d]^{C_{ij}}\ar@{^{(}->}[r]&V_{ij}\times\II\ar[d]^{B_{ij}}\ar@{^{(}->}[r]&U_{ij}\times\II\ \ar[d]^{A_{ij}}\\
W_{ij}\times\II\ar@{^{(}->}[r]&V_{ij}\times\II\ar@{^{(}->}[r]&U_{ij}\times\II.
}
\end{displaymath}
This correspond to sections $D_{ij}\in \Gamma(U_{ij},T_S\otimes \mathcal{I}_{y})$, $d_{ij}\in \Gamma(V_{ij},T_X\otimes \mathcal{I}_{y})$ such that $D_{ij}\mid_X=d_{ij}$. This are sections $D_{ij}$ of $T_S\langle X\mid y_1,\ldots,y_l\rangle$. To check the cocycle condition and obstruction space is the same as without the markings, we refer to \cite[Prop. 1.2.9 and Prop. 1.2.12]{S} for details. The vanishing of $H^0$ follows from the vanishing of $H^0(T_S)=0$ and the inclusion in (\ref{ex.seq.2}). The vanishing of $H^2$ follows from the inclusion 
$$T_S\langle X\mid y_1,\ldots,y_l\rangle\subset T_S\langle X\rangle.$$
The cokernel of this inclusion is supported on the points $y_1,\ldots,y_l$ and $H^2(T_S\langle X\rangle)=0$ in the established range, cf. \cite[Prop. 4.8]{FlKPS}.
\end{proof}

We can give an alternative proof for the dominance of the normalization map. Let 
$$(S,X,x_1,\ldots,x_\delta,y_1,\ldots,y_l)\in \mathcal{V}_{g,\delta,l}$$
be a general point and $f:C\to X$ the normalization. We call 
$$p+q=p_1+q_1+\ldots+p_\delta+q_\delta$$ 
the preimage of the nodes and $y=y_1+\ldots+y_l$ the marked points.

\begin{proposition}
The differential of the normalization map 
$$dc_{g,\delta,l}:T\mathcal{V}_{g,\delta,l}\to T\mathcal{M}_{g-\delta,[2\delta]+l}$$
at a point $(S,H,X,x_1,\ldots,x_\delta,y_1,\ldots,y_l)\in \mathcal{V}_{g,\delta,l}$ can be identified with $H^1$ of the map coming from (\ref{ex.seq.1});
$$H^1\left(S,T_S\langle X\mid y_1,\ldots,y_l\rangle\right)\to H^1\left(X,T_X(-y)\right).$$
The identification on the right is given by the isomorphism induced by $f_*$
$$H^1\left(C,T_C(-(p+q)-y)\right)\overset{\sim}{\to} H^1\left(X,f_*\left(T_C(-(p+q)-y)\right)\right)\cong H^1(X,T_X(-y)).$$
\end{proposition}
\begin{proof}
From the proof of Proposition \ref{prop}, one can see that the map induced by restriction sends locally trivial first order deformations of $(S,X,y)$ to locally trivial first order deformations of $(X,y)$. On the other hand, around each node $X$ is given by $\{xy=0\}\subset U\subset \af^2$, $\Omega_U$ is generated by $dx, dy$ with the relation $xdy+ydx=0$ and $T_X=\Hom(\Omega_X, \mathcal{O})$ restricted to $U$ is generated by the maps $(dx,dy)\mapsto (x,0)$ and $(dx,dy)\mapsto(0,y)$. On the branch defined by $x$, the first generator is $x$ times the generator $dx\mapsto 1$ of the tangent bundle at that branch and the same for $y$. This means that locally around the node the normalization map gives us an isomorphism
$xT_{U_x}\oplus yT_{U_y}\cong T_U.$
These local isomorphisms coincide away from the nodes forming a global isomorphism  
$$f_*T_C(-(p+q))\cong T_X.$$
After tensoring with $\mathcal{O}(-y)$, the composition of this with the natural isomorphism 
$$H^1(C,T_C(-(p+q)-y))\cong H^1(X,f_*T_C(-(p+q)-y))$$
identifies first order deformations of $(X,y)$ preserving the nodes with deformations of $(C,y)$ together with the marked points $p+q$.
\end{proof}

As a corollary we have:

\begin{corollary}
The normalization map $c_{g,\delta,l}$ is dominant for $g\geq 3$, $\delta\leq g-2$ and $g\leq 11$ with $g\neq 10$.
\end{corollary}

\begin{proof}
By the exact sequence (\ref{ex.seq.1}) and Proposition \ref{prop}, the cokernel of $dc_{g,\delta,l}$ is isomorphic to 
$$H^2(S,T_S(-X)\otimes \mathcal{I}_y).$$
On the other hand, the cokernel of the inclusion $T_S(-X)\otimes \mathcal{I}_y\subset T_S(-X)$ is supported on the points $y$. Thus, 
$$H^2(S,T_S(-X)\otimes \mathcal{I}_y)\cong H^2(S,T_S(-X))\cong H^0(S,\Omega_S^1(X)).$$
The dimension of the last vector space for general $(S,\mathcal{O}_S(X))\in \mathcal{F}_g$ is zero when $g\neq 10$ and $g\leq 11$, cf. \cite[\S 5.2]{Be}.
\end{proof}

Recall that if $f:C\to X\subset S$ is the normalization map and $\mathcal{N}_f$ is the normal sheaf of $f$ defined by the exact sequence
$$0\to T_C\to f^*T_{S}\to N_f\to 0,$$ 
then, cf. \cite[Lemma 4.16]{FlKPS},

\begin{align}
\label{push}
f_*N_f\cong \mathcal{N}'_{X/S}
\end{align}
and from the exact sequence above, since $S$ is a K3 surface, $N_{f}\cong \omega_C$. Thus, $h^1(\mathcal{N}'_{X/S})=1$ and the kernel of the surjection coming from the exact sequence (\ref{ex.seq.0})
$$\gamma:H^1(S,T_S)\to H^1(X,\mathcal{N}'_{X/S})$$
is $19$-dimensional. It can be thought as the tangent space of $\mathcal{F}_g$ at $(S,\mathcal{O}_S(X))$ and 
$$0\to H^0(\mathcal{N}'_{X/C})\to H^1(T_S\langle X\rangle)\to \ker(\gamma)\to 0$$
as the natural differential sequence
$$0\to T_{[X]}\left(V_\delta(S,H)\right)\to T_{(S,X)}\mathcal{V}_{g,\delta}\to T_{(S,H)}\mathcal{F}_g\to 0.$$
Here $H=\mathcal{O}_S(X)$, see \cite[Prop. 4.8]{FlKPS}. Let $\eps:\tilde{S}\to S$ be the blow up of $S$ at the marked points $x_1,\ldots,x_\delta,y_1,\ldots,y_l$ and 
$$E=E_1+\ldots+E_\delta, \hspace{0.2cm} F=F_1+\ldots+F_l$$
the exceptional divisors corresponding to the nodes and marked points respectively. The proper transform $C$ of $X$ lies in the linear system $|\eps^*H-2E-F|$ and the following sequence is exact:
\begin{equation}
\label{ex.seq2}
\begin{aligned}
0\to \sh{F}_C(-E-F)\to\eps^*T_S(-E-F)\to N_f(-(p+q)-y)\to 0.\\
\end{aligned}
\end{equation}

\begin{lemma} 
\label{lem}
In the same setting as above
$$\eps_*\left(\sh{F}_C(-F)\right)\cong T_S\langle X\mid y_1,\ldots,y_l\rangle,\hspace{0.5cm} R^{i}\eps_*\left(\sh{F}_C(-E-F)\right)=R^{i}\eps_*\left(\sh{F}_C(-F)\right)$$
and the latter vanishes for $i>1$. Moreover, $H^1(\sh{F}_C(-E-F))\cong H^1(\sh{F}_C(-F))$.
\end{lemma}
\begin{proof}
There are maps as in the following diagram:
\begin{displaymath}
\xymatrix{
0\ar[r]&T_S\langle X\mid y_1,\ldots,y_l\rangle\ar[d]\ar[r]& T_S\otimes\mathcal{I}_{y}\ar[r]\ar[d]^{\cong}&\mathcal{N}'_{X/S}(-y)\ar[d]^{\cong}\ar[r]&0\\
0\ar[r]&\eps_*\left(\sh{F}_C(-F)\right)\ar[r]&\eps_*\left(\eps^*T_S(-F)\right)\ar[r]&\eps_*N_f(-y)\ar[r]&0.
}
\end{displaymath}
The lower row is the pushforward by $\eps$ of the upper row in (\ref{diagP5}), after twisting by $\mathcal{O}(-F)$. The map in the middle comes from the isomorphism $\eps_*\left(\mathcal{O}(-F)\right)\cong \eps_*\eps^*\mathcal{I}_y$ and the map on the right comes from the isomorphism (\ref{push}). This induces an isomorphism 
$$T_S\langle X\mid y_1,\ldots,y_l\rangle\cong \eps_*\left(\sh{F}_C(-F)\right).$$
This implies that $R^1\eps_*\left(\sh{F}_C(-F)\right)=0$.The rest follows from the fact that, cf. \cite[Prop. 1.9]{Ba},  
\begin{equation}
\label{FmidE}
\sh{F}_C\otimes\mathcal{O}_E\cong\sh{F}_C(-F)\otimes\mathcal{O}_E\cong \mathcal{O}_E(-1)^{\oplus 2}.
\end{equation} 
This, together with the exact sequence
$$0\to\sh{F}_C(-E-F)\to\sh{F}_C(-F)\to\mathcal{O}_E(-1)^{\oplus 2}\to0,$$
give us that $R^i\eps_*\left(\sh{F}_C(-F)\right)\cong R^i\eps_*\left(\sh{F}_C(-E-F)\right)$ and its vanishing for $i\geq 1$. As a consequence we have that $H^1(\sh{F}_C(-E-F))\cong H^1(\sh{F}_C(-F))$.
\end{proof}

Now consider the diagram coming from (\ref{ex.seq2}):

\begin{align}
\label{seq2}
\xymatrix{
H^1\left(\tilde{S},\sh{F}_C(-E-F)\right)\ar[r]&H^1\left(\tilde{S},\eps^*T_S(-E-F)\right)\ar[r]^{\alpha}\ar[dr]_{\beta}&H^1\left(C, N_f(-(p+q)-y)\right)\ar[d]\\
&&H^1\left(C,N_f\right).
}
\end{align}

\begin{proposition}
\label{diff.map}
The tangent space of $\mathcal{F}_{g,\delta+l}$ at $(S,H,x_1,\ldots,x_\delta,y_1,\ldots,y_l)$ can be identified with $\ker(\beta)$ and the differential of the map $\pi:\mathcal{V}_{g,\delta,l}\to\mathcal{F}_{g,\delta+l}$ is given by 
$$d\pi: H^1\left(S,\sh{F}_C(-E-F)\right)\to \ker(\alpha)\subset\ker(\beta).$$
In particular $\pi$ is dominant (on a component) when, for some point $(S,X,x_1,\ldots,x_\delta,y_1,\ldots,y_l)$, the map induced by inclusion
$$H^1\left(C,N_f(-(p+q)-y)\right)\to H^1\left(C,N_f\right)$$
is an isomorphism.
\end{proposition}

As a corollary we have:
\begin{corollary}
\label{dom}
In the range $3\leq g\leq 11$, $1\leq \delta\leq g-2$ and $g\neq 10$, the moduli map 
$$\pi:\cal{V}_{g,\delta,l}\to\cal{F}_{g,\delta+l}$$ is dominant when $3\delta+l\leq g$.
\end{corollary}

\begin{proof}
Recall that in the established range the map $\mathcal{V}_{g,\delta,l}\to\mathcal{M}_{g-\delta,[2\delta]+l}$ is dominant. In particular, for a general point $(S,X,x_1,\ldots,x_{\delta},y_1,\ldots,y_l)\in\mathcal{V}_{g,\delta,l}$, the points lying over the nodes and the marked points
$$p+q=(p_1+q_1)+\ldots+(p_\delta+q_\delta),\hspace{0.5cm} y= y_1+\ldots+y_l$$
are general in the symmetric product $\hbox{Sym}^{2\delta+l}C$. Since $S$ is a $K3$ surface, $N_{f}\cong \omega_C$ and the surjection 
$$H^1\left(C,N_f\left(-\sum \tilde{y}_i-E\mid_C\right)\right)\to H^1\left(C,N_f\left(-\sum \tilde{y}_i\right)\right)$$
is an isomorphism if and only if $h^0(C,\mathcal{O}_C((p+q)+y))\leq1$. This holds when $2\delta+l\leq g-\delta$.
\end{proof}
Now we prove the main proposition.
\begin{proof}[Proof of Proposition \ref{diff.map}]
We call $x=x_1+\ldots+x_\delta$ and $y=y_1+\ldots+y_l$ the nodes and marked points in $S$. The map $\eps:\tilde{S}\to S$ is birational and finite when restricted to $C$. Thus,  
$$R^i\eps_*\mathcal{O}_{\tilde{S}}\cong R^i\eps_*\mathcal{O}_{\tilde{S}}(-E-F)\cong 0\hbox{ for }i>0,$$  
and there is an isomorphism coming from the Leray spectral sequence sending isomorphically the map $\beta$ on the right triangle in (\ref{seq2}) to
$$H^1\left(S,T_S\otimes \mathcal{I}_{x+y}\right)\overset{\beta}{\to}H^1\left(X,\mathcal{N}'_{X/S}\right).$$
This map is the composition 
$$\beta:H^1(S,T_S\otimes \mathcal{I}_{x+y})\to H^1(S,T_S)\to H^1(\mathcal{N}'_{X/S}).$$
Thus, elements of $\ker(\beta)$ can be interpreted as first order deformations of the pointed surface $(S,x_1,\ldots,x_\delta,y_1,\ldots,y_l)$ such that, after forgetting the marked points, they lie in the kernel of $H^1(T_S)\to H^1(\mathcal{N}'_{X/S})$, namely, deformations that preserve the genus $g$ marking. This proves the first assertion.  

Lemma \ref{lem} and pushing forward by $\eps$ makes the $H^1$-map of (\ref{ex.seq2}) become isomorphic to 
$$H^1(S,T_S\langle X\mid y_1,\ldots,y_l\rangle) \to H^1(S,T_S \otimes \mathcal{I}_{x+y}).$$
By Proposition \ref{prop}, we have our result.

\end{proof}

\section{Parameter spaces of flags and degeneration to projective cones}
The goal of this section is to show that the general fiber of the normalization map 
$$c_{g,\delta,l}:\mathcal{V}_{g,\delta,l}\to\mathcal{M}_{g-\delta,[2\delta]+l}$$
is irreducible when $g\leq 11$ and $g\neq 10$.\\

Notice that for a scheme $S$ and $S$-schemes $X,Y,Z$, if $X\to Y$ is a birational $S$-map, the same holds for the induced map
$$X\times_SZ\to Y\times_SZ.$$
Therefore, it is enough to prove that the general fiber of 
$$c_{11,\delta}:\cal{V}_{11,\delta}\to\cal{M}_{11-\delta,[2\delta]}$$
is irreducible. It is shown in the proof of \cite[Thm. 2.1]{CD} that the general fiber of 
$$\mathcal{V}_{g,\delta}\big/\Sigma_\delta\to\mathcal{M}_{g-\delta}$$
is irreducible for $\delta$ and $g$ in certain range. This result can be easily extended to the fibers of $c_{g,\delta}$. We extend the argument of \cite[\S 2.1]{CD}, following the same exposition.\\

Let $\hbox{\textbf{F}}_g$ be the component of the Hilbert scheme whose general point parameterizes primitive $K3$ surfaces in $\pp^g$ of degree $2g-2$. The group $\hbox{PGL}(g+1)$ acts on $\hbox{\textbf{F}}_g$, the dimension is $19+(g+1)^2-1$ and the quotient map induces a rational map defined on the locus of smooth K3 surfaces
$$\hbox{\textbf{F}}_g\dashrightarrow\cal{F}_g.$$
This map is, over an open set, a $\hbox{PGL}(g+1)$-bundle. Let $\hbox{\textbf{Hilb}}_g$ be the component of the Hilbert scheme of curves on $\pp^g$ whose general point parameterizes canonical curves lying on a hyperplane of $\pp^g$. The quotient, defined on the open locus of nodal curves,
$$\hbox{\textbf{Hilb}}_g\dashrightarrow\Mgbar,$$ 
is an open subset of a $(\pp^g)^\vee\times\hbox{PGL}(g)$-bundle. We also define $\hbox{\textbf{V}}_g$ to be the component of the flag Hilbert scheme, cf. \cite{Kl}, whose general point is a pair $(S,X)$ with $S$ a general point on $\hbox{\textbf{F}}_g$ and $X$ an hyperplane section of $S$. An open subset of $\hbox{\textbf{V}}_g$ is a $\pp^g$-bundle over $\hbox{\textbf{F}}_g$. Notice that 
$$\dim\left(\hbox{\textbf{Hilb}}_g\right)= 4g-4+g^2\hspace{0.5cm}\hbox{and}\hspace{0.5cm}\dim\left(\hbox{\textbf{V}}_g\right)=18+g+(g+1)^2.$$
The natural forgetful maps sit in the following diagram
\begin{displaymath}
\xymatrix{
&\hbox{\textbf{V}}_g\ar[rd]^{p}\ar[dl]&\\
\hbox{\textbf{F}}_g&&\hbox{\textbf{Hilb}}_g.
}
\end{displaymath}
Our situation is slightly different from \cite{CD}, since we need to keep track of the nodes. For $0\leq \delta\leq g-2$, we define the incidence variety
$$\hbox{\textbf{V}}_{g,\delta}\subset \hbox{\textbf{V}}_g\times (\pp^g)^{\delta}$$
to be the closure of the set of points $(S,X,x_1,\ldots,x_\delta)$ with $X$ irreducible $\delta$-nodal with nodes at $x_1,\ldots,x_\delta$. We denote by $\hbox{\textbf{Nod}}_{g,\delta}$ the closure of the  locally closed subset of $\hbox{\textbf{Hilb}}_g\times(\pp^g)^{\delta}$ consisting of irreducible $\delta$-nodal curves $X$ together with points $(x_1,\ldots,x_\delta)\in(\pp^g)^{\delta}$ such that $X$ is nodal at $x_1,\ldots,x_{\delta}$. The first thing to notice is that $\hbox{\textbf{Nod}}_{g,\delta}$ it is irreducible, since the rational map induced by normalization
$$\begin{array}{rcl}
\hbox{\textbf{Nod}}_{g,\delta}&\dashrightarrow&\mathcal{M}_{g-\delta}\\
(X,x_1,\ldots,x_\delta)&\mapsto&C
\end{array}$$ 
is dominant and the fiber over a general canonical curve $C$ is, up to projective transformations, birationally 
$$(\hbox{Sym}^2C)^{\times \delta}.$$
Notice that $\dim\hbox{\textbf{Nod}}_{g,\delta}=g^2+4g-4-\delta$. Now consider the forgetful map 
$$p_{g,\delta}:\hbox{\textbf{V}}_{g,\delta}\to\hbox{\textbf{Nod}}_{g,\delta}.$$
Notice that 
$$\dim\hbox{\textbf{V}}_{g,\delta}=g^2+3g+19-\delta,$$
since for $0\leq \delta\leq g-2$, $\dim V_\delta(S,H)=\dim|H|-\delta$. One can also observe that the dimension of the general fiber of $p_{g,\delta}$ is at least $23-g$.\\

The irreducibility of the fiber relies on two facts. The first was proved by \cite{P} and states that a smooth $K3$ surface $S\subset \pp^g$ can flatly degenerate inside $\pp^g$ to a projective cone over any hyperplane section $S_X\subset \pp^g$. Moreover, (see \cite[Lemma 2.3]{CD}) one can do this inside the fiber of $p_{g,\delta}$. Thus, if $S_X$ is the projective cone over the nodal curve $X\subset S$ inside $\pp^g$, then $(S_X,X,x_1,\ldots,x_{\delta})$ lies in every irreducible component of the general fiber of $p_{g,\delta}$. \\

The second ingredient is the smoothness of the fiber at the point $(S_X,X,x_1,\ldots,x_{\delta})$. If so, then irreducibility of the general fiber $p_{g,\delta}$ follows and therefore the irreducibility of the general fiber for the induced map on the quotient $c_{g,\delta}$. \\

The map 
$$\begin{array}{rcl}
\hbox{\textbf{V}}_{g,\delta}&\to&\hbox{\textbf{V}}_{g,\delta}\big/\Sigma_\delta\\
(S,X,x_1,\ldots,x_\delta)&\mapsto&(S,X)
\end{array}$$
is \'{e}tale and the tangent space of the fiber at $(S_X,X)$ (see \cite[\S 4.5.2]{S} and \cite[Lemma 2.4]{CD}) is isomorphic to 
$$H^0(S_X,N_{S_X/\pp^g}(-1))\cong \bigoplus_{k\geq 1}H^0(X,N_{X/\pp^g}(-k)).$$ 
The computation of these cohomology groups was done in \cite[\S 3]{CM} by specializing to a \textit{canonical graph curve}, namely, the union of $2g-2$ lines in $\pp^{g-1}$, each meeting three others at distinct points. The dual graph of such a curve is a trivalent graph with $2g-2$ nodes and $3g-3$ edges. For such a curve $\Gamma_g$, when $3\leq g\leq11$ and $g\neq 10$, 
$$\bigoplus_{k\geq 1}H^0(\Gamma_{11},N_{\Gamma_{11}/\pp^{11}}(-k))=H^0(\Gamma_{11},N_{\Gamma_{11}/\pp^{11}}(-1))=23-g.$$
Every irreducible $\delta$-nodal curve degenerates to such a graph curve inside $\pp^{g-1}$, see \cite[Prop 2.6]{CD}. By upper semi-continuity of $h^0(S_X,N_{S_X/\pp^g}(-1))$ one shows that the dimension of the general fiber of $p_{g,\delta}$ is at most $23-g$. Recall that the previous calculations gave us that the general fiber is al least $23-g$. Thus, the general fiber is smooth at $(S_X,X)$ and this point lies in every irreducible component. From this follows the irreducibility of the general fiber of $p_{g,\delta}$ when $3\leq g\leq 11$ and $g\neq 10$. The irreducibility of the general fiber of $p_{g,\delta}$ carries down to the quotient $c_{g,\delta}$. We have proved birationality in Theorem \ref{thm2}. 

\section{The birational type of $\mathcal{F}_{11,n}$}
Theorem \ref{thm2} and Corollary \ref{dom} give us a birational map
$$c_{11,\delta,l}:\mathcal{V}_{11,\delta,l}\to\mathcal{M}_{11-\delta,[2\delta]+l}$$
and a dominant map $\mathcal{V}_{11,\delta,l}\to\mathcal{F}_{11,\delta+l}$, for $3\delta+l\leq 11$.
 
\begin{proof}[Proof of Theorem \ref{thm}]
The moduli space $\mathcal{M}_{11-\delta,[2\delta]+l}$ is the $\zz_2^{\oplus \delta}$-quotient of $\mathcal{M}_{11-\delta,2\delta+l}$ and it is known \cite[Thm 7.1]{L} that $\mathcal{M}_{9,n}$ is unirational for $n\leq 8$ and therefore its quotient. In particular $\mathcal{V}_{11,2,4}$ is unirational and dominates $\mathcal{F}_{11,6}$. On the other hand $\mathcal{M}_{9,n}$ is uniruled for $n\leq 10$, cf. \cite[Thm. 0.6]{FV1}. Thus, the moduli space $\mathcal{V}_{11,2,5}$ is uniruled and the map $\mathcal{V}_{11,2,5}\to\mathcal{F}_{11,7}$ is dominant and generically finite. 
\end{proof}

\begin{remark}
Notice that with $\delta=1$ and $l=7$, one has a dominant rational map $\mathcal{M}_{10,9}\dashrightarrow \mathcal{F}_{11,8}$ where the source is uniruled, cf. \cite[Prop. 7.6]{FP}. It would be enough to show that the relative dimension one map $\mathcal{V}_{11,1,7}\to\mathcal{F}_{11,8}$, does not contract  the covering rational curves of $\mathcal{M}_{10,9}$. Unfortunately, by a similar argument as in the next lemma, one can see that the fibers of the map are indeed rational curves. 
\end{remark}

We also have:

\begin{lemma}
For $g\leq 11$, $g\neq 10$, and $3\delta+l=g$, the map
$$\pi:\mathcal{V}_{g,\delta,l}\to\mathcal{F}_{g,\delta+l}$$
is birational. 
\end{lemma}

\begin{proof}
Let $(S,H, x_1,\ldots,x_\delta,y_1,\ldots,y_{l})$ be a general point in $\mathcal{F}_{g,\delta,l}$. By Corollary \ref{dom}, the fiber of $\pi$ over the point $(S,H, x_1,\ldots,x_\delta,y_1,\ldots,y_{l})$ corresponds to a subset of the linear system
$$|H\otimes\frak{m}_{x_1}^2\otimes\ldots\otimes\frak{m}_{x_\delta}^2\otimes\frak{m}_{y_1}\otimes\ldots\otimes\frak{m}_{y_l}|\subset|H|$$
containing a non-empty and zero-dimensional open, whose general element is nodal. Here $\frak{m}_x$ stands for the maximal ideal of the point $x$. Thus, the general fiber consist of a single point. 
\end{proof}

As a consequence we have Theorem \ref{bir}.

\begin{proof}[Proof of Theorem \ref{bir}]
By the lemma above, the map $\mathcal{V}_{11,\delta,l}\to\mathcal{F}_{11,\delta+l}$ is birational when $3\delta+l=11$. Composing with the already constructed birational map 
$$\mathcal{V}_{11,\delta,l}\to\mathcal{M}_{11-\delta,[2\delta]+l}$$
gives us our result.
\end{proof}

In particular $\mathcal{V}_{11,1,8}\to\mathcal{F}_{11,9}$ is birational. This gives us a birational map 
$$\mathcal{M}_{10,10}\big/\zz_2\overset{\sim}{\dashrightarrow} \mathcal{F}_{11,9}.$$

\begin{proof}[Proof of Theorem \ref{thm3}]
It is known (cf. \cite[Thm. 0.1]{FV1}) that the Kodaira dimension of the universal Jacobian over $\mathcal{M}_{10}$
$$\hbox{Kod}\left(\overline{\sh{J}ac}_{10}^{10}\right)=0.$$
There is a generically finite rational map
$$\overline{\mathcal{M}}_{10,10}\big/\zz_2\to \overline{\mathcal{M}}_{10,10}\big/\Sigma_{10}\overset{\sim}{\dashrightarrow}\overline{\sh{J}ac}_{10}^{10},$$
where $\Sigma_{10}$ is the symmetric group on $10$ letters. Thus, 
$$\hbox{Kod}(\overline{\mathcal{M}}_{10,10}\big/\zz_2)\geq 0.$$
For $n\geq 9$, the result follows from the inequality $\hbox{Kod}(\Fgn)\leq \hbox{Kod}(\mathcal{F}_{g,n+1})$.
\end{proof}


\end{document}